\documentclass[12pt]{article}

\usepackage{amsmath, amssymb, amsthm}
\usepackage{graphicx}
\usepackage{hyperref}
\usepackage{geometry}
\usepackage{enumitem}
\usepackage{mathrsfs}
\usepackage{setspace}

\usepackage{graphicx}%
\usepackage{multirow}%
\usepackage{amsmath,amssymb,amsfonts}%
\usepackage{amsthm}%
\usepackage{mathrsfs}%
\usepackage[title]{appendix}%
\usepackage{xcolor}%
\usepackage{textcomp}%
\usepackage{manyfoot}%
\usepackage{booktabs}%
\usepackage{listings}%
\usepackage{esint}

\newtheorem{theorem}{Theorem}
\newtheorem{proposition}[theorem]{Proposition}%
\newtheorem{lemma}{Lemma}

\newtheorem{remark}{Remark}%
\newtheorem{assumption}{Assumption}

\newtheorem{definition}{Definition}%

\def\R{\mathbb R}
\def\N{{\mathbb N}}

\newcommand{\Lip}{\operatorname{Lip}}
\newcommand{\lip}{\operatorname{lip}}
\def\BL{{\mathsf{BL}}} 
\def\RCD{\mathsf{RCD}}

\def\d{{\partial}}

\def\i{{\infty}}

\DeclareMathOperator{\supp}{supp}

\DeclareMathOperator{\dist}{d}

\allowdisplaybreaks

\title{A step towards the tensorization of Sobolev spaces}

\author{Silvia Ghinassi\thanks{Math Department, Shoreline Community College, Shoreline, 98133, WA, USA, email: sghinassi@shoreline.edu}, Vikram Giri \thanks{Department of Mathematics, ETH, Zurich, Switzerland, email: vikramaditya.giri@math.ethz.ch}, Elisa Negrini \thanks{Department of Mathematics, University of California Los Angeles, Los Angeles, 90024, CA, USA, email: enegrini@math.ucla.edu},}
\date{}

\begin{document}

\maketitle

\begin{abstract}
We prove that Sobolev spaces on Cartesian and warped products of metric spaces tensorize, only requiring that one of the factors is a doubling space supporting a Poincar\'e inequality.
\end{abstract}


\section{Introduction}\label{sec1}

The theory of first order Sobolev spaces in abstract metric measure spaces became an independent field of study the late 90s. Haj\l asz \cite{hajlasz1996} defined the concept of Haj\l asz gradient and subsequently the so-called Sobolev-Haj\l asz spaces.
Heinonen and Koskela introduced the concept of upper gradient to replace the distributional gradient in more abstract spaces,  \cite{HK1995,HK1996,HK1998}. Their definition requires an inequality to hold true for every rectifiable curve in the underlying space. Shanmungalingam \cite{shanmugalingam2000newtonian} introduced the concept of weak upper gradients, and the relative Newton-Sobolev spaces by allowing the requirement to fail for a null set of curves (in the sense of modulus of a curve family). In his pioneering work on calculus on metric measure spaces \cite{cheeger1999}, Cheeger introduced a relaxation procedure of Heinonen and Koskela's notion of weak upper gradients.

More recently, in the interest of studying spaces that satisfy a lower Ricci curvature bound, Lott and Villani \cite{LV2007} and (independently) Sturm \cite{sturm2006one,sturm2006two} introduced the celebrated curvature-dimension condition. In order to further study these spaces, Ambrosio, Gigli, and Savar\`e \cite{AGS2014RCD} introduced the so-called RCD spaces and a new notion of first order Sobolev spaces, suitable for a larger class of metric measure spaces. There has been a flurry of activity in the last few years, due to the interest in studying these general metric spaces with a lower curvature bound. See the brief surveys~\cite{honda, naber} for more information.

Under sufficient geometric assumptions, the notions of weak upper gradient due to Cheeger, Shanmungalingam, and Ambrosio-Gigli-Savar\`e are equivalent \cite[Theorem~7.4]{AGS2013}, while Haj\l asz gradients are not; one should think of them as the Hardy–Littlewood maximal function of upper gradients. 

The tensorization problem asks what is the relation between the Sobolev functions on two metric measure spaces $X$ and $Y$ and the ones on their product space (equipped with the product measure, and Cartesian product metric). The problem, while basic, seems to not have been explored until recently; in fact it is still open for the general case. Ambrosio, Gigli, and Savar\`e \cite[Theorem~6.13]{AGS2014RCD}, first proved the result under some restrictive condition on the curvature of \emph{both} spaces (namely the $\RCD(K,\infty)$ condition). Ambrosio, Pinamonti, and Speight \cite[Theorem~3.4]{APS2015} proved it for spaces that are doubling and support a Poincar\'e inequality, or under certain quadracity assumptions. Ambrosio, Gigli, and Savar\`e \cite{AGS2014RCD} proved that, for a general metric space, it is always true that we can control the Euclidean product of the weak upper gradients in each factor by the weak upper gradient on the product.

In this paper we prove that indeed we can control the Euclidean product of the weak upper gradients in each factor by the weak upper gradient on the product, provided \emph{one} of the factors is doubling and has a Poincar\'e inequality. We also prove that the same is true if one considers warped products of metric measure spaces. Both results, Theorem \ref{Thm3.7} and Theorem \ref{Thm5}, and their proofs, generalize the work of Gigli and Han \cite[Theorem~3.7 and Theorem~3.12]{gigli2018sobolev}, where the authors proved the conjecture when one of the factors is an interval, both for Cartesian and warped products. We remark that the proof we present is almost self contained, relying only on earlier results concerning the density of Lipschitz functions by Ambrosio, Gigli, and Savar\`e \cite{AGS2014}.

Independently, Eriksson-Bique, Soultanis, and Rajala \cite[Theorem~1.6]{eriksson2024tensorization} proved similar results to ours, but only in the Cartesian case. Their proof relies on techniques developed by the first and third author \cite{eriksson2021curvewise}.
The same authors \cite[Theorem~1.2]{EBRS2022hilbert} then also proved tensorization where both factors are quasi-Hilbertian metric spaces. 

\section{Preliminaries}

For basic definitions on metric spaces, we refer the interested reader to \cite{heinonen}. We start with a few definitions and results, which we will need to introduce the notion of Sobolev functions (Definition \ref{def:sobolev}).

\bigskip \begin{definition}[Absolutely continuous curves]
    Let $p\in[1,+\infty]$ and $\gamma \in C([0,1], X)$ and indicate $\gamma_t = \gamma(t)$ for $t\in [0,1]$. We say that $\gamma \in AC^p([0,1], X)$ if there exists $G \in L^p([0,1])$ such that
    \begin{equation}\label{distC}
        d(\gamma_s,\gamma_t) \leq \int_s^t G(r) \,dr \quad \forall t,s \in [0,1], \; s<t.
    \end{equation}
    For $p=1$ the space $AC^1([0,1], X)$ is denoted $AC([0,1], X)$ and it is the space of absolutely continuous curves.
\end{definition}

\bigskip \begin{theorem}[Theorem 1.1.2 in \cite{ambrosio2005gradient} ]
    For $\gamma\in AC([0,1], X)$ there exists an a.e. minimal function $G$ that satisfies (\ref{distC}) called the metric derivative, and it can be computed for a.e. $t\in[0,1]$ as
    $$|\dot{\gamma}_t|:= \lim_{s\rightarrow t} \frac{d(\gamma_t, \gamma_s)}{|s-t|}.$$
\end{theorem}

\bigskip \begin{proposition}[\cite{burago2022course}]
    The length of a curve $\gamma \in AC([0,1], X)$ is given by:
$$l[\gamma] := \int_0^1 |\dot{\gamma}_t| \,dt.$$
If $(X,d)$ is a length space then for any $x,y\in X$
$$\dist(x,y) = \inf \left\{\; \int_0^1 |\dot{\gamma}_t| \,dt;\quad \gamma \in AC([0,1], X) \text{ connecting } x \text{ and } y\; \right\}.$$
\end{proposition}    

\bigskip \begin{definition}[Metric doubling]
Given a measure space $(X,d, \mu)$, we say that $X$ is \emph{metric doubling} if there exists a constant $C_X>1$ such that, for all $r>0$ every ball of radius $r$ can be covered by $C_X$ balls of radius $r/2$.
\end{definition}

\bigskip \begin{definition}[Measure doubling]
    A measure space $(X, d, \mu)$ is said to be \emph{measure doubling} if there exists a constant $D_X>1$ such that, for all $x \in X$ and $r>0$ we have $\mu(B(x,2r)) \leq D_X \mu(B(x,r))$.
\end{definition}

\bigskip \begin{remark}\label{doub_const}
Every metric space $X$ that carries a doubling measure with constant $D_X$ is metric doubling with constant $C_X$, with $C_X \leq D_X^4$. For a proof of this fact we refer the reader to \cite[Section~4.1]{heinonen2015sobolev}. 
\end{remark}

\bigskip \begin{definition}[Local and global Lipschitz constants]
Given $f \colon X \to \R$, the local Lipschitz constant of $f$ is the function $\lip_X(f):X\rightarrow [0,+\infty]$ defined as
\begin{align}
 \lip_X(f)(x) =
\begin{cases}
     \limsup_{y \to x} \frac{|f(x)-f(y)|}{\dist_X(x,y)} &\quad \text{if } x \text{ is not isolated},\\
     0 &\quad \text{otherwise}.
\end{cases}
\end{align}
Analogously, the global Lipschitz constant is defined as
\begin{equation}
    \Lip(f) = \limsup_{y \neq x} \frac{|f(x)-f(y)|}{\dist_X(x,y)},
\end{equation}
and if $X$ is a length space, $\Lip(f)=\sup_x\lip_X(f)(x)$.
We denote by $\Lip(X)$ the space of all Lipschitz functions $f \colon X \to \R$.
\end{definition}

\bigskip \begin{definition}[Test plan]
Let $(X,d,m)$ be a metric measure space and $\pi$ a probability measure on $C([0,1],X)$. The measure $\pi$ is said to have bounded compression if there exists a constant $C>0$ such that for all $t\in[0,1]$
$$(e_t)_{\#}\pi \leq Cm,$$
where the evaluation map $e_t$ is given by $e_t(\gamma)= \gamma_t$.
$\pi$ is said to be a test plan if it has bounded compression, it is concentrated on $AC^2([0,1], X)$ and 
$$\int_0^1 \; \int\; |\dot{\gamma}_t|^2 \;d\pi(\gamma)\;dt < +\infty.$$
\end{definition}

\bigskip \begin{definition}[Sobolev class]\label{def_sob}
Let $(X,d,m)$ be a metric measure space. A Borel function $f:X \rightarrow\mathbb{R}$ belongs to the Sobolev class $S^2(X,d,m)$ (respectively $S_{loc}^2(X,d,m)$ ) if there exists a non-negative function $G\in L^2(X,m)$ (respectively  $G\in L_{loc}^2(X,m)$) such that for all test plans $\pi$
$$\int \; |f(\gamma_1)-f(\gamma_0)| \; d\pi(\gamma) \leq \int\;\int_0^1 G(\gamma_s) |\dot{\gamma}_s| \; ds\; d\pi(\gamma).$$
In this case the function $G$ is called a 2-weak upper gradient of $f$, or simply a weak upper gradient of $f$.
\end{definition}

\bigskip \begin{remark}\label{rem_sob}
Among all weak upper gradients of $f$ there exists a minimal function $G$ in the $m$-a.e. sense. Such minimal function is called minimal weak upper gradient and we denote it by $|Df|$. Notice that if $f$ is Lipschitz, then $|Df| \leq \lip_X(f)$ $m$-a.e. since $\lip_X(f)$ is a weak upper gradient of $f$.

We will use that minimal weak upper gradients are lower semicontinuous in the following sense: if $f_n \in S^2(X,d,m)$ is a sequence converging $m$-a.e. to some $f$ such that the sequence given by the functions $|Df_n|$ is bounded in $L^2(X,m)$, then $f \in  S^2(X,d,m)$ and for all $G$ that are the $L^2$-weak limit of some subsequence of $|Df_n|$ we have 
\begin{equation}\label{lsc}
    |Df|\leq G.
\end{equation} 
Finally, we will later use the fact that space $S^2_{loc}(X,d,m) \cap L^{\infty}_{loc}(X,d,m)$ is an algebra and the following inequalities hold:
    \begin{equation}\label{ProdRule}
        |D(fg)| \leq |f||Dg| + |g||Df| \quad m\text{-a.e., for all $f,g \in S^2_{loc}(X,d,m) \cap L^{\infty}_{loc}(X,d,m)$}. 
    \end{equation}
    \begin{equation} \label{sublinear}
    |D(\alpha f + \beta g)| \leq |\alpha| |Df| + |\beta||Dg| \quad m\text{-a.e., for all $f,g \in S^2_{loc}(X,d,m)$ and $\alpha, \beta \in \R$}.
    \end{equation}

For additional details on the properties of minimal weak upper gradients, see \cite{AGS2014}.
\end{remark}

\bigskip \begin{definition}[Sobolev space] \label{def:sobolev}
The Sobolev space $W^{1,2}(X,d,m)$ is defined as $$W^{1,2}(X,d,m) := S^{2}(X,d,m) \cap L^2(X,m).$$
The space $W^{1,2}(X,d,m)$ endowed with the norm
$$\|f\|_{W^{1,2}(X,d,m)} := \sqrt{\|f\|_{L^2(X,m)}^2 +\| |Df| \|_{L^2(X,m)}^2} $$
is a Banach space.\footnote{Note that in general it is not an Hilbert space.}  
\end{definition}

\bigskip \begin{lemma}[Density in energy of Lipschitz functions, \cite{AGS2014}] \label{densitylip}
Let $Y$ be a complete and separable metric measure space, and $f \in W^{1,2}(Y)$. Then there exists a sequence of Lipschitz functions $f_n$ that converges to $f$ in $L^2$ and such that $\lip_Y(f_n)$ converges in $L^2$ to $|Df|$.
\end{lemma}

\bigskip \begin{assumption}\label{XAssum}
We will have, unless otherwise specified, the following set of assumptions for the metric measure space $(X,\dist_X,m_X)$:
\begin{itemize}
    \item $(X,\dist_X)$ is a complete and separable length space,
    \item $m_X$ is a non-negative Borel measure with respect to $\dist_X$ and it is finite on bounded sets,
    \item $\supp(m_X) = X$.
\end{itemize}
In the following we may denote this space simply by $X$.
\end{assumption}

\bigskip \begin{assumption}\label{YAssum}
We will have, unless otherwise specified, the following set of assumptions for the metric measure space $(Y,\dist_Y,m_Y)$:
\begin{itemize}
    \item $(Y,\dist_Y,m_{Y})$ is complete, and $C_Y$-measure doubling length space,
    \item $m_Y$ is a non-negative Borel measure with respect to $\dist_Y$ and it is finite on bounded sets,
    \item $\supp(m_Y) = Y$,
    \item $Y$ supports a $(2,2)$-Poincar\'e inequality, that is, for every $r >0$, there exists constants $\lambda\geq1$ and $C_p$, such that for any metric ball $B \subset Y$ of radius smaller than $r$, we have
\[
\fint_B |u-u_B|^2 \,dm_Y \leq C_P \operatorname{rad}(B)^2 \int_{\lambda B} |g|^2,
\]
where $g$ is any weak upper gradient for $u$, and $u_B = \fint_B u \,dm_Y$.
\end{itemize}
In the following we may denote this space simply by $Y$.
\end{assumption}

\bigskip \begin{remark}
Note that since $Y$ is measure doubling, it is also separable, see \cite[Lemma~4.1.13]{heinonen2015sobolev}. The most restrictive assumption (on $Y$) is the existence of a Poincar\'e inequality. This fact will be used exclusively to obtain (\ref{riemann}) in Lemma \ref{second}.
\end{remark}

\section{Cartesian products}

The product space $X \times Y$ is given the Euclidean product metric 
\begin{equation}
    \dist_{X \times Y} ((x,t),(y,s)) := \sqrt{ \dist_X(x,y)^2 + \dist_Y(t,s)^2 },
\end{equation}
and the measure on it is the usual product measure and {is denoted simply by $m$}.

\bigskip \begin{definition} [Beppo Levi space]\label{def.BL}
The Beppo-Levi space $\BL(X,Y)$ is the space of functions $f \in L^2(X \times Y;\R)$ such that 
\begin{enumerate}
    \item $f(x,\cdot) \in W^{1,2}(Y)$ for $m_X \text{ - a.e. } x$
    \item $f(\cdot,t) \in W^{1,2}(X)$ for $m_{Y} \text{ - a.e. } t$
    \item the function
    \begin{equation}
        |Df|_{\BL}(x,t) := \sqrt{|Df(x,\cdot)|_Y(t)^2 + |Df(\cdot,t)|_X(x)^2}
    \end{equation}
    belongs to $L^2(X \times Y;\R).$
\end{enumerate}
The Beppo-Levi norm is defined as
\begin{equation}
    \|f\|_{\BL} := \sqrt{\|f\|_{L^2}^2 + \| |Df|_{\BL} \|_{L^2}^2}\,.
\end{equation}
\end{definition}

\bigskip

We will use the following property of minimal weak upper gradients on Cartesian products. If $f \colon X \times Y \to \R$, then for any $F \subset Y$, 
\begin{equation}\label{contsublin}
\left|D \left(\int_F f(x,y)\,dy\right)\right|_X \leq \int_F \left|Df(x,y)\right|_X \,dy.
\end{equation}
In fact, this follows from Definition~\ref{def_sob} and Remark~\ref{rem_sob}, since for any test plan $\pi$ on $C([0,1],X)$ we have 
\begin{align*}
    &\int \left|\int_F f(\gamma_1,y)\,dy - \int_F f(\gamma_0,y)\,dy \right|\, d\pi(\gamma) = \int \left|\int_F f(\gamma_1,y) - f(\gamma_0,y)\,dy \right|\, d\pi(\gamma)\\
    &\leq \int \int_F \left|f(\gamma_1,y) - f(\gamma_0,y)\right| \,dy \, d\pi(\gamma)   \leq \int_F \int \int_0^1 |Df(\gamma_s,y)|_X |\dot\gamma_s| \, ds\, d\pi(\gamma)\,dy\\
    &\leq \int \int_0^1 \left(\int_F |Df(\gamma_s,y)|_X \,dy\right) |\dot\gamma_s| \, ds\, d\pi(\gamma)\,,
\end{align*}
so $\int_F |Df(x,y)|_X \,dy$ is a weak upper gradient for $\int_F f(x,y)\,dy$ and \eqref{contsublin} consequently follows.

\begin{remark}
    In this section, we will be considering functions on the Cartesian product space $X \times Y$. To avoid cumbersome notation, from now on we will denote by $|Df|$ the minimal weak upper gradient on $X \times Y$, while we will write $|Df(x,\cdot)|_Y(t)$ as $|{\d f}/{\d t}|(x,t)$, and $|Df(\cdot,t)|_X(x)$ as $|{\d f}/{\d x}|(x,t)$. 
\end{remark}

\medskip
Our main theorem for this section is: 

\bigskip \begin{theorem}\label{Thm3.7}
The sets $W^{1,2}(X \times Y)$ and $\BL(X,Y)$ coincide and for every $f \in W^{1,2}(X \times Y)=\BL(X,Y)$ we have
\begin{equation} \label{main}
|Df|_{\BL} \leq |Df| \leq C_0 |Df|_{\BL} \quad \text{$m_{X \times Y}$-a.e.,}
\end{equation}
where $C_0$ is a constant that depends only on $C_Y$ and $C_P$.
\end{theorem}

\bigskip \begin{proposition}[Prop. 6.18 in \cite{AGS2014RCD}] \label{oneinclusion}
We have $W^{1,2}(X \times Y) \subset \BL(X,Y)$ and 
\begin{equation}
|Df|_{\BL} \leq |Df|  \quad \text{$m_{X \times Y}$-a.e.,}
\end{equation}
\end{proposition}

In light of the proposition, it is enough to show that $\BL(X,Y) \subset W^{1,2}(X \times Y)$ by establishing that the second inequality in (\ref{main}) holds. To prove this we will need the following lemmas.

\bigskip \begin{lemma}\label{lem.lip}
Let $N>0$ be a fixed natural number. Let $f: X\times Y \to \R$ be of the form $f(x,t) = \sum_{i=1}^N h_i(t)g_i(x)$ where $g_i \in \Lip(X)$ and $h_i \in \Lip(Y)$ for all $1\leq i \leq N$. Then 
\begin{equation}
    \lip_{X\times Y}(f)^2(x,t) \leq 2\left(\lip_X(f(\cdot,t))^2(x) + \lip_Y(f(x,\cdot))^2(t)\right).
\end{equation}
for every $(x,t) \in X\times Y$
\end{lemma}

This lemma replaces Lemma 3.3 in \cite{gigli2018sobolev}. The key difference between the two is the use of triangle inequality instead of Cauchy-Scwhartz inequality.

\begin{proof}
We have
{
\begin{align*}
    & \lip_{X\times Y}(f)^2(x,t) = \limsup_{(y,s)\to(x,t)} \frac{|f(y,s)-f(x,t)|^2}{\dist_{X\times Y} ((y,s),(x,t))^2} \\
    &\leq 2 \limsup_{(y,s)\to(x,t)} \frac{1}{\dist_{X\times Y} ((y,s),(x,t))^2}\left( \dist_X(y,x)^2 \frac{|f(y,s)-f(x,s)|^2}{\dist_X(y,x)^2} + \dist_Y(s,t)^2 \frac{|f(x,s)-f(x,t)|^2}{\dist_Y(s,t)^2} \right) \\
    &\leq 2 \limsup_{(y,s)\to(x,t)} \frac{\left|\sum_{i=1}^N h_i(s)(g_i(y)-g_i(x))\right|^2}{\dist_X(y,x)^2} + 2\limsup_{(y,s)\to(x,t)} \frac{\left|\sum_{i=1}^N (h_i(s)-h_i(t))g_i(x)\right|^2}{\dist_Y(s,t)^2} \\
    &\leq  2\lip_X(f(\cdot,t))^2(x) + 2\lip_Y(f(x,\cdot))^2(t),
\end{align*}
where in the last step we used the continuity of $h_i$.
}
\end{proof}

In order to proceed, we need to introduce a suitable partition on $Y$. One option would be to use generalization of dyadic cubes in metric measure spaces (such as the ones constructed in \cite{christ1990}). However, since we will not be using the tree-like structure, we can work with a simpler construction, such as the one in the following lemma.

\bigskip \begin{lemma}[Lemma 37 in \cite{ambrosio2015sobolev}]\label{cubes}
For every $k \in \N$ there exists a collection of open subsets of $Y$, $Q_{i,k}$ and points $t_{i,k}$ (the ``centers'' of the ``cubes''), and $i \in I_k$, where $I_k$ is a countable set, such that
\begin{itemize}
\item  $m_Y\left(Y \setminus \bigcup_i Q_{i,k}\right) = 0$ for all $k \in \N$; 
\item  For every $i,j \in I_k$ either $Q_{j,k} = Q_{i,k}$ or $Q_{j,k}\cap Q_{i,k} = \varnothing$ (i.e. the sets $Q_{i,k}$ form a partition);
\item if $i \neq j$, then $d_Y(t_{i,k},t_{j,k}) > \frac{1}{k}$;
\item each $Q_{i,k}$ is comparable to a ball centered at $t_{i,k}$ of radius roughly $\frac{1}{k}$, 
\[B \left(t_{i,k}, \frac{1}{3}\frac{1}{k}\right) \subset Q_{i,k} \subset B \left(t_{i,k}, \frac{5}{4}\frac{1}{k}\right).
\]
\end{itemize}
\end{lemma}

\bigskip \begin{remark}\label{rem.qik}
    We say that $Q_{i,k}$ and $Q_{j,k}$ are neighbors, and we write $Q_{i,k} \sim Q_{j,k}$, if their distance is less than $\frac{1}{k}$. If that is the case, then it must be that $d_Y(t_{i,k}, t_{j,k}) \leq \frac{4}{k}$. Given $Q_{i,k} \sim Q_{j,k}$, we have that their centers are at most $\frac{10}{4}\frac{1}{k} + \frac{1}{k} = \frac{14}{4}\frac{1}{k}\leq \frac{4}{k}$ where $\frac{5}{4k}$ comes from the fact that both cubes are contained in a ball with radius $\frac{5}{4k}$ while the $\frac{1}{k}$ is from the definition of $Q_{i,k} \sim Q_{j,k}$. This implies that each $Q_{i,k}$ has at most $C_Y^3$ neighbors. In fact, we can cover $B(t_{i,k}, \frac{4}{k})$ with $C_Y^3$ balls of radius $\frac12\frac{1}{k}$, but the condition $d_Y(t_{i,k},t_{j,k}) > \frac{1}{k}$ means each ball can only contain one of the points $t_{i,k}$.
\end{remark}

\bigskip \begin{definition}
    
For every $k \in \N$, let $\{\chi_{i,k}\}_{i\in I_k}$ be a (fixed) partition of unity subordinate to the partition of $Y$ by ``cubes'' $Q_{i,k}$ as above. That is, $\chi_{i,k}=1$ on $B \left(t_{i,k}, \frac{1}{3}\frac{1}{k}\right)$, and $\supp \chi_{i,k} \subset B \left(t_{i,k}, \frac{5}{4}\frac{1}{k}\right)$. Each $\chi_{i,k}$ is $c_1k$-Lipschitz, where $c_1=c_1(C_Y)$. See for instance the construction in \cite{kinnunenlatvala2022}. 
\end{definition}
\bigskip

Now, we are ready for our main technical lemma.
\begin{lemma} \label{second}

There exists a constant $C_1>0$ depending only on $C_Y$ and $C_P$ so that for any $f \in \BL(X,Y)$ there exists a sequence $(f_k)$ such that $f_k \in \Lip(X\times Y) \cap \BL(X,Y)$ with $f_k \to f$ in $L^2(X\times Y)$, $|Df_k|_{\BL}$ uniformly bounded in $L^2(X\times Y)$, and for any weak limit $G$ of $|Df_k|_{\BL}$ in $L^2(X\times Y)$, we have 
\begin{equation}\label{e.sec}
    G \leq C_1 |Df|_{\BL}\,.
\end{equation}
\end{lemma}

\begin{proof}
By approximating $f$ with $\max\{\min\{f,\lambda\},-\lambda\}$ for $\lambda>0$, we can assume that $f$ is bounded. Moreover, by multiplying with a sequence of Lipschitz functions supported on an exhausting sequence of balls in $X \times Y$, and employing \eqref{ProdRule} to ensure convergence in the Beppo Levi space, we can also assume that $f$ has bounded support contained in $B((x,y),R)$, for some $(x,y) \in X\times Y$ and $R>1$. Given $f \in \BL(X,Y)$, we define a sequence of functions $F_k$ as 
\begin{equation}\label{def.fn}
    F_k (x,t) := \sum_{i \in I_k} \chi_{i,k}(t) f_{k,i}(x),
\end{equation}
where $f_{k,i}(x) := \fint_{Q_{i,k}} f(x,t)\,dt$.
For a fixed $(x,t)$ at most $C_Y^3$ terms in the sum are non-zero since, by Lemma~\ref{cubes}, the support of $\chi_{i,k}$ can intersect only the support of the neighbours of the $Q_{i,k}$ and, by Remark~\ref{rem.qik}, each $Q_{i,k}$ has at most $C_Y^3$ neighbors. Moreover, it is enough to consider cubes that intersect $\supp f$, and since they are bounded sets, $\supp f$ and all cubes that intersect it are contained in $B(y,2R)$, for some $y \in Y$.

Moreover, by Jensen's inequality,
\begin{align*}
    \|f_{k,i}\|_{L^2(X)}^2 &= \int_{X} \left|\fint_{Q_{ik}}f(x,t)\,dm_Y(t)\right|^2\,dm_X(x)\\ &\leq \int_{X} \fint_{Q_{ik}}|f(x,t)|^2\,dtdx = \fint_{Q_{ik}}\int_{X} |f(x,t)|^2\,dxdt \leq \frac{\|f\|_{L^2(X \times Q_{ik})}^2}{m_Y(Q_{ik})} \,
\end{align*}
and so
\begin{align*}
        \|F_k\|_{L^2(X \times Y)}^2 &\lesssim_{C_Y} \sum_{i \in I_k} \left\|\chi_{i,k}(t) f_{k,i}(x)\right\|^2_{L^2(X \times Y)} \\
        & \lesssim_{C_Y} \sum_{i \in I_k} m_Y\left(B\left(t_{i,k}, \frac54 \frac1k\right)\right)\left\|f_{k,i}\right\|^2_{L^2(X)}\\
        &\lesssim_{C_Y, D_Y} \sum_{i \in I_k} \|f\|^2_{L^2(X \times Q_{ik})} 
         \lesssim_{C_Y, D_Y} \;\|f\|^2_{L^2(X \times Y)}\,.
\end{align*}
Thus, we see that the linear map $T_k \colon L^2(X \times Y) \to L^2(X \times Y)$  that takes $f$ with bounded support to $F_k$ is Lipschitz, uniformly in $k$. Note now that, if $g \in \Lip (X \times Y)$ with bounded support then $T_k(g) \to g$ in $L^2(X\times Y)$, since then
\begin{align*}
        \|T_k(g)-g\|_{L^2(X \times Y)}^2 &\lesssim_{C_Y} \sum_{i \in I_k} \left\|\chi_{i,k}(t) \left|\fint_{Q_{ik}}\left(g(x,s)-g(x,t)\right)\,dm_Y(s)\right|\right\|^2_{L^2(X \times Y)} \\
        &\lesssim_{C_Y} \sum_{\substack{i \in I_k \\ Q_{i,k}\cap \supp g \neq \emptyset}} \Lip_Y(g)^2\left\|\frac{\chi_{i,k}(t)}k \right\|^2_{L^2(X \times Y)}\to 0\,.
\end{align*}
Thus, because $\Lip (X \times Y)$ is dense in $L^2(X \times Y)$, see Proposition 4.3 in \cite{AGS2013}, we get 
\begin{equation}\label{l2conv}
    F_k \to f \qquad \text{in} \quad L^2(X \times Y).
\end{equation}

Note now that $f_{k,i} \in W^{1,2}(X)$, since by convexity and \eqref{contsublin},

\begin{align*}
    |\partial f_{k,i}/\partial x|^2 &= \left|\frac{\partial}{\partial x}\left(\fint_{Q_{ik}}f(x,t)\,dm_Y(t)\right)\right|^2 {\leq }\left( \fint_{Q_{ik}} \left|\partial f(x,t)/\partial x\right| \,dm_Y(t)\right)^2\\
    &\leq \fint_{Q_{ik}} |\partial f(x,t)/\partial x|^2 \,dm_Y(t) \leq \frac{\||Df|_{BL}\|^2_{L^2(X\times Y)}}{m_Y(Q_{ik})}
\end{align*}

Thus, by Lemma \ref{densitylip}, we have a sequence of Lipschitz functions $ f_{k,i}^n \colon X \to \R$ such that $f_{k,i}^n \to f_{k,i}$ in $L^2(X,m_X)$ and $\lip_X( f_{k,i}^n) \to |D  f_{k,i}|_X$ in $L^2(X,m_X)$. Now define
\begin{equation}
    F_k^n (x,t) := \sum_{i \in I_k} \chi_{i,k}(t) f_{k,i}^n(x).
\end{equation}


Let $n(k)$ be the least number $n$ such that, for all $i$, we have

\begin{equation} \label{epsilon} 
\|f^{n(k)}_{k,i} - f_{k,i}\|_{L^2(X)} + \|\lip_X f^{n(k)}_{k,i} - |\partial f_{k,i}/\partial x|\|_{L^2(X)} \leq \frac{1}{k^{3}m_{X\times Y}(B((x,y),2R))} .
\end{equation}

Now, we define $f_k := F_k^{n(k)}$. Observe that $f_k\in \Lip(X\times Y) \cap \BL(X,Y)$ and note also that we get
\begin{equation}\label{l2conv-again}
    f_k \to f \qquad \text{in} \quad L^2(X \times Y).
\end{equation}

It remains to show \eqref{e.sec}. Since $m_X \times m_Y$ is a Borel measure and any open set in $X\times Y$ can be written as a countable union of sets of the form $E\times F$ where $E\subset X$ and $F\subset Y$ are open, by lower semi-continuity it suffices to show that for every $k \in \mathbb{N}$ and every such open subsets $E,F$, 
\begin{equation} \label{notFTC}
    \limsup_{k\to\infty}\int_{E \times F} |\lip_X f_k|^2(x,t) \, dm(x,t) \lesssim \int_{E \times F} |{\d f}/{\d x}|^2(x,t) \, dm_Y(t) \, dm_X(x).
\end{equation}
and that 
\begin{equation} \label{FTC}
    \limsup_{k\to\infty}\int_{E \times F} |\lip_Y f_k|^2(x,t) \, dm(x,t) \lesssim  \int_{E \times F} |{\d f}/{\d t}|^2(x,t) \, dm_Y(t) \, dm_X(x).
\end{equation}
To prove \eqref{notFTC}, we first notice that by convexity and \eqref{contsublin}, we have
\begin{align*}
    \int_E |\d f_{k,i} / \d x|^2(x) dm_X(x) \leq \int_E \fint_{Q_{i,k}} |\d f/\d x|^2(x,s) dm_Y(s) dm_X(x).
\end{align*}
Also, by sublinearity of $\lip_X$
\begin{align*}
    |\lip_X f_k|^2(x,t) &= \left|\lip_X \left(\sum_{i=0}^\infty \chi_{i,k}(t) f^{n(k)}_{i,k}(x) \right)\right|^2 \leq \left(\sum_{i=0}^\infty \chi_{i,k}(t) \left| \lip_Xf^{n(k)}_{i,k}\right|(x) \right)^2 \\
    &\lesssim_{C_Y} \sum_{i=0}^\infty \chi_{i,k}(t) |\lip_X f^{n(k)}_{i,k}|^2(x)\,.
\end{align*}

Thus, using sublinearity again, we have\\
\begin{align*}
    \int_{E \times F} |\lip_X f_k |^2(x,t) \, dm(x,t) &= \int_E \int_F |\lip_X f_k |^2(x,t) \, dm_Y(t)\,dm_X(x) \\
    &\lesssim_{C_Y} \int_E \int_F \sum_{i \in I_k} \chi_{i,k}(t) |\lip_X f^{n(k)}_{i,k}|^2(x)\, dm_Y(t)\,dm_X(x)\\
    &\leq \int_E \int_F \sum_{i \in I_k} \chi_{i,k}(t) |\d f_{i,k}/\d x|^2 (x) \, dm_Y(t)\,dm_X(x) + \frac1{k^{3}} \\
    &\lesssim \int_E \sum_{i \in I_k} \int_F \chi_{i,k}(t)\, dm_Y(t)\, |\d f_{i,k}/\d x|^2(x) \,dm_X(x) + \frac1{k^{3}}\\
    &\lesssim_{D_Y} \int_E \sum_{i \in I_k} |Q_{i,k}| \frac1{|Q_{i,k}|}\int_{Q_{i,k}} |\d f/\d x|^2(x,s)\, dm_Y(s)\, dm_X(x) + \frac1{k^{3}}\\
    &= \int_{E \times B_k(F)} |{\d f}/{\d x}|^2(x,t) \, dm_Y(t) \, dm_X(x) + \frac1{k^{3}} \,,
\end{align*}
where in the second inequality we have used (\ref{epsilon}) and $B_k(F)$ denotes the union of all the $Q_{ik}$ for $i\in I_k$ such that $Q_{ik}\cap F\neq\emptyset$. Now sending $k\to\infty$ and noting that $F$ is open gives us~\eqref{notFTC}.
 
It remains to show (\ref{FTC}). First, recall that $\chi_{j,k}(t) = 1 - \sum_{i \neq j} \chi_{i,k}(t)$ for $j \in \N$, and so we have that, for any $t, s \in Y$, and $j \in \N$,
\begin{align*}
    & F_k^n (x,t) - F_k^n (x,s) = \sum_{i \in I_k} \chi_{i,k}(t) f_{k,i}^n(x) - \sum_{i \in I_k} \chi_{i,k}(s) f_{k,i}^n(x) \\
     &  = \sum_{\substack{i \in I_k \\ i\neq j}} \chi_{i,k}(t) f_{k,i}^n(x) + \chi_{j,k}(t)f_{k,j}^n(x) - \chi_{j,k}(s)f_{k,j}^n(x) - \sum_{\substack{i \in I_k \\ i\neq j}} \chi_{i,k}(s) f_{k,i}^n(x) \\ 
     & = \sum_{\substack{i \in I_k \\ i\neq j}} \chi_{i,k}(t) f_{k,i}^n(x) - \sum_{\substack{i \in I_k \\ i\neq j}} \chi_{i,k}(t)f_{k,j}^n(x) +\sum_{\substack{i \in I_k \\ i\neq j}} \chi_{i,k}(s)f_{k,j}^n(x) - \sum_{\substack{i \in I_k \\ i\neq j}} \chi_{i,k}(s) f_{k,i}^n(x) \\
    &=\sum_{\substack{i \in I_k \\ i\neq j}} \left(\chi_{i,k}(t) - \chi_{i,k}(s)\right) (f_{k,i}^n(x) - f_{k,j}^n(x)).
\end{align*}

Note that the sum above is locally finite, as only finitely many of the $\chi_{i,k}$ are nonzero for any fixed values of $t,s \in Y$. Assume $t \in Q_{j,k}$.
Dividing this by $t-s$, and then taking a limit we obtain 
\begin{align*}
     \lip_Y(F_k^n(x))(t) & =  \limsup_{s \to t} \left| \sum_{i \neq j} \frac{\chi_{i,k}(t) - \chi_{i,k}(s)}{t-s} (f_{k,i}^n(x) - f_{k,j}^n(x)) \right| \\
    & \leq \sum_{\substack{i \in I_k \\i \neq j \\ Q_{i,k} \sim Q_{j,k}}}  \Lip_Y(\chi_{i,k})\left|f_{k,i}^n(x) - f_{k,j}^n(x)\right|\\
    & \lesssim_{c_1} \sum_{\substack{i \in I_k \\i \neq j \\ Q_{i,k} \sim Q_{j,k}}}  k \left|f_{k,i}^n(x) - f_{k,j}^n(x)\right|\,.
\end{align*}
Set $B_{jk} := B(t_{j,k},\frac{6}{k})$, and note that if $Q_{i,k}$ and $Q_{j,k}$ are neighbors, then they are both contained in $B_{j,k}$. Defining $f_{B_{jk}} := f_{B_{jk}}(x)= \fint_{B_{jk}} f(x,t)\,dm_Y(t)$, we have
\begin{align*}
        & \left| \fint_{Q_{i,k}} f(x,t)\,dm_Y(t) - \fint_{Q_{j,k}} f(x,t)\,dm_Y(t)\right|^2 \\
        & = \left| \fint_{Q_{i,k}} (f(x,t)- f_{B_{jk}})\,dm_Y(t) - \fint_{Q_{j,k}} (f(x,t)- f_{B_{jk}})\,dm_Y(t)\right|^2 \\
        & \leq 2 \left( \left| \fint_{Q_{i,k}} (f(x,t)- f_{B_{jk}})\,dm_Y(t)\right|^2 + \left|\fint_{Q_{j,k}} (f(x,t)- f_{B_{jk}})\,dm_Y(t)\right|^2 \right)\\
        & \leq 2\left(\frac{1}{|Q_{i,k}|} \int_{Q_{i,k}} \left|f(x,t)- f_{B_{jk}}\right|^2\,dm_Y(t)  + \frac{1}{|Q_{j,k}|} \int_{Q_{j,k}} \left|f(x,t)- f_{B_{jk}}\right|^2\,dm_Y(t)\right) \\
        & \leq 2\left(\frac{1}{|Q_{i,k}|} \int_{B_{jk}} \left|f(x,t)- f_{B_{jk}}\right|^2\,dm_Y(t)  + \frac{1}{|Q_{j,k}|} \int_{B_{jk}} \left|f(x,t)- f_{B_{jk}}\right|^2\,dm_Y(t)\right) \\
        &\lesssim_{D_Y} \fint_{B_{jk}} \left|f(x,t)- f_{B_{jk}}\right|^2\,dm_Y(t).
\end{align*}

Hence, using Poincar\'e inequality, putting the above estimates together, we obtain
\begin{align} 
    &\int_{E \times F}  |\lip_Y f_k |^2(x,t) \, dm(x,t)
     {\lesssim_{C_Y,c_1}} \int_E \sum_{j} \int_{Q_{j,k}} k^2 \sum_{\substack{i \neq j \\ Q_{i,k} \sim Q_{j,k}}}  |f^{n(k)}_{k,j} (x) - f^{n(k)}_{k,i}(x)|^2 \,dm_Y(t)\,dm_X(x)\notag\\
    & \lesssim_{C_Y,c_1} \int_E \sum_{j} \int_{Q_{j,k}}\,dm_Y(t)\sum_{\substack{i \neq j \\ Q_{i,k} \sim Q_{j,k}}} k^2 \left| \fint_{Q_{i,k}} f(x,s)\,dm_Y(s) - \fint_{Q_{j,k}} f(x,s)\,dm_Y(s)\right|^2 \,dm_X(x) + \frac1{k^{}}\notag\\
    & \lesssim_{C_Y,c_1,D_Y} \int_E \sum_{j} |Q_{j,k}|\sum_{\substack{i \neq j \\ Q_{i,k} \sim Q_{j,k}}} k^2  \fint_{B_{jk}} \left|f(x,s)- f_{B_{jk}}\right|^2\,dm_Y(s)\,dm_X(x) + \frac1{k^{}}\notag \\
    & \lesssim_{C_Y,c_1,D_Y,C_P} \int_E \sum_{j} |Q_{j,k}|\sum_{\substack{i \neq j \\ Q_{i,k} \sim Q_{j,k}}} k^2  \operatorname{rad}({B_{jk}})^2 \fint_{\lambda {B_{jk}}} |\partial f/\partial s|^2\,dm_Y(s)\,dm_X(x) + \frac1{k^{}}\label{riemann}\\
    & \lesssim_{C_Y,c_1,D_Y,C_P} \int_E \sum_{j} |Q_{j,k}| \fint_{\lambda {B_{jk}}} |\partial f/\partial s|^2\,dm_Y(s)\,dm_X(x) + \frac1{k^{}}\notag\\
    & \lesssim_{C_Y,c_1,D_Y,C_P,\lambda}  \int_E \sum_{j} \int_{\lambda {B_{jk}}} |\partial f/\partial t|^2\,dm_Y(t)\,dm_X(x) + \frac1{k^{}}\notag\\
    & \lesssim_{C_Y,c_1,D_Y,C_P,\lambda} \int_{E \times \tilde B_k(F)} |\partial f/\partial t|^2\,dm(x,t) + \frac1{k^{}}\,,\notag
\end{align}
where $\tilde B_k(F)$ denotes the union of all the $\lambda B_{jk}$ such that $Q_{jk}\cap F\neq\emptyset$. Now, as before, taking the $\limsup$ for $k \to \infty$ concludes the proof. 
\end{proof}

We are now ready to prove the main theorem.

\begin{proof}[Proof of Theorem \ref{Thm3.7}]
We already observed that $W^{1,2}(X \times Y) \subset \BL(X,Y)$. Let $f \in \BL(X,Y)$ and let $\{f_k\} \subset \BL(X,Y) \cap \Lip(X\times Y)$ be as in Lemma \ref{second}. Lemma \ref{lem.lip} says that
\[
 \lip_{X\times Y} (f_k)^2 \leq 2\left(\lip_X (f_k)^2 + \lip_Y (f_k)^2\right)
\]

By Lemma \ref{second}, since $f_k \to f$ in $L^2$, the lower semicontinuity of weak upper gradients (\ref{lsc}) implies that $f \in W^{1,2}(X \times Y)$ and
\[
|Df|_{X \times Y} \leq G
\]
where $G$ is any weak limit of $\lip_{X\times Y} (f_k)$. So,
\[
|Df|_{X \times Y} \leq 2C_1 |Df|_{\BL},
\]
which together with Proposition \ref{oneinclusion} concludes the proof. 
\end{proof}

\section{Warped products}

\bigskip \begin{definition}
Let $(X,d_X)$ and $(Y,d_Y)$ be length spaces, and $w_d \colon Y \to [0,\infty)$ a continuous function. Let $\gamma = (\gamma^X,\gamma^Y)$ be a curve such that $\gamma^X$ and $\gamma^Y$ are absolutely continuous. Then the $w_d$-length of $\gamma$ is defined as
\[
\ell_w(\gamma) = \lim_{\tau} \sum_{i=1}^n \sqrt{d_Y^2(\gamma^Y_{t_{i-1}},\gamma^Y_{t_{i}}) + w_d^2(\gamma^Y_{t_{i-1}})d_X^2(\gamma^X_{t_{i-1}},\gamma^X_{t_{i}})},
\]
where $\tau$ is a partition of $[0,1]$ and the limit is taken over refinement ordering of partitions.
\end{definition}
The limit exists and 
\[
\ell_w(\gamma) = \int_0^1 \sqrt{|\dot{\gamma}^Y_t|^2 + w_d^2(\gamma_t^Y)|\dot{\gamma}^X_t|^2}\,dt.
\]

\bigskip \begin{definition}
Let $(X,d_X)$ and $(Y,d_Y)$ be length spaces, and $w_d \colon Y \to [0,\infty)$ a continuous function. We define a pseudo-metric $\dist_w$ on on the space $X \times Y$ by 
\[
    \dist_w(p,q) = \inf\{\ell_w(\gamma) \mid  \text{$\gamma^X \in AC([0,1],X)$, $\gamma^Y\in AC([0,1],Y)$, and $\gamma_0=p, \gamma_1=q$}\},
\]
for any $p,q \in X \times Y$.
\end{definition}
\bigskip
The pseudo-metric induces an equivalent relation on $X \times Y$ given by $(x,y) \sim (x',y')$ if $d_w((x,y),(x'y'))=0$ and hence a metric on the quotient. We denote the completion of such quotient by $(X \times_w Y,d_w)$. If both $X$ and $Y$ are separable, so is $X \times_w Y$. Let $\pi \colon X \times Y \to X \times_w Y$ be the quotient map.

\bigskip \begin{definition}
Let $(X,d_X,m_X)$ and $(Y,d_Y,m_Y)$ be complete separable and length metric spaces equipped with non-negative Radon measures. Assume that $m_X(X) < \infty$ and let $w_d,w_m \colon Y \to [0,\infty)$ be continuous functions. Then the warped product $(X \times_w Y, d_w)$ is defined as above and the Radon measure $m_w$ is defined as
\[
m_w =\pi_*((w_m m_Y)\times m_X).
\]
\end{definition}
Note that the assumption that $m_X$ is a finite measure is needed to ensure that $m_w$ is Radon (it is always Borel). See \cite{gigli2018sobolev} after Definition 2.9 for more details. Following the notation conventions in \cite{gigli2018sobolev}, with a slight abuse, we will denote an element of $X \times_w Y$ by $(x,y)$.

\bigskip \begin{definition}
As a set, the Beppo Levi space $\BL_w(X,Y)$ is the subset of $L^2(X \times_w Y, m_w)$ of all functions $f$ such that
\begin{itemize}
    \item for $m_X$-a.e. $x \in X$, we have $f^{(x)}=f(x,\cdot) \in W^{1,2}(Y,w_m m_Y)$;
    \item for $w_m m_Y$-a.e. $t \in Y$, we have $f^{(t)}=f(\cdot, t) \in W^{1,2}(X)$;
    \item the function 
    \[
    |Df|_{\BL_w} = \sqrt{w_d^{-2}|Df^{(t)}|_X^2(x) + |Df^{(x)}|_Y^2(t)}
    \]
    belongs to $L^2(X \times_w Y, m_w)$.
\end{itemize}
On $\BL_w(X,Y)$ we put the norm
\[
\|f\|_{\BL_w(X,Y)} = \sqrt{\|f\|_{L^2}^2 + \||Df|_{\BL_w}\|_{L^2}^2}.
\]
\end{definition}

To handle the warped case we need to introduce an auxiliary space:
\bigskip \begin{definition}\label{Vset}
Let $\mathcal{V} \subset \BL_w(X,Y)$ be the space of functions $f$ which are identically $0$ on $X \times \Omega$, where $\Omega$ is an open set that satisfies $\{w_m = 0\} \subset \Omega \subset Y$. $\BL_{0,w}(X,Y) \subset \BL_w(X,Y)$ is defined as the closure of $\mathcal{V}$ in  $\BL_w(X,Y)$.
\end{definition}
\bigskip
We want to compare the Beppo-Levi space with the Sobolev space on the warped product, and the respective notions of minimal upper gradients. The main result of this section is the following.

\begin{theorem}\label{Thm5}
    Let $X$ and $Y$ satisfy Assumption \ref{XAssum} and \ref{YAssum}, and let $w_d, w_m \colon Y \rightarrow [0, \infty)$ be continuous functions such that $\{w_d = 0\}\subset\{w_m = 0\}$. Then 
    \[
\BL_{0,w}(X,Y) \subset W^{1,2}(X \times_w Y)   \subset  \BL_w(X,Y),
\]
and, for every $f \in W^{1,2}(X \times_w Y) \subset \BL_w(X,Y)$, the inequalities
\[
   |Df|_{\BL_w} \leq |Df|_{X \times_w Y}\leq C_0 |Df|_{\BL_w}
\]
hold $m_w$-a.e., with $C_0>0$ is as in Theorem \ref{Thm3.7}.
\end{theorem}
 \bigskip

The proof of Theorem \ref{Thm5} is decomposed into Propositions \ref{prop6}, \ref{Prop312}, \ref{prop8} below.
We will continue denoting the minimal weak upper gradient on the Cartesian product $X \times Y$ simply by $|Df|$, while we will denote the minimal weak upper gradient on the warped product by $|Df|_{X \times_w Y}$.

\bigskip \begin{proposition} \label{prop6}
    We have $ W^{1,2}(X \times_w Y) \subset \BL_w(X,Y)$.
\end{proposition}
\begin{proof}
Let $f\in W^{1,2}(X \times_w Y)$. Then by Lemma \ref{densitylip}, we can find a sequence $f_n$ of Lipschitz functions on $X \times_w Y$ such that $f_n \rightarrow f$ and $\lip_{X \times_w Y}(f_n) \rightarrow |Df|_{X \times_w Y}$ in $L^2(X \times_w Y)$.  Passing to a subsequence, which we do not relabel, we can assume that $\|f_n-f\|_{L^2(X \times_w Y,m_w)}<n^{-2}$. Then

\begin{align*}
    &\left\|\sum_{n=1}^{\infty} \left\|f_n^{(t)} - f^{(t)}\right\|_{L^2(X,m_X)}\right\|_{L^2(Y,w_m m_Y)}\\
   &  \leq \sum_{n=1}^{\infty} \left\|\left\|f_n^{(t)} - f^{(t)}\right\|_{L^2(X,m_X)}\right\|_{L^2(Y,w_m m_Y)}\\
   & = \sum_{n=1}^{\infty} \left\|f_n-f \right\|_{L^2(X \times_w Y,m_w)} < \infty
\end{align*}

Passing to a subsequence, which we do not relabel, we can assume that $\|f_n-f\|_{L^2(X \times_w Y,m_w)}<n^{-4}$. Then
\begin{align*}
    &\left\|\sum_{n=1}^{\infty} \left\|f_n^{(t)} - f^{(t)}\right\|_{L^2(X,m_X)}\right\|_{L^2(Y,w_m m_Y)}\\ 
    &= \left(\int_Y\; \left( \sum_{n=1}^{\infty} \frac{n}{n} \left\|f_n^{(t)} - f^{(t)}\right\|_{L^2(X,m_X)} \right)^2\, w_m(t)dm_Y(t)\right)^{1/2}\\
    &\leq \text{(Cauchy-Schwarz)}  \left(\int_Y\; \left(\sum_{n=1}^{\infty} \frac1{n^2}\right) \left(\sum_{n=1}^{\infty} n^2 \left\|f_n^{(t)} - f^{(t)}\right\|_{L^2(X,m_X)}^2 \right)\, w_m(t)dm_Y(t)\right)^{1/2}\\
    &=  \left(\sum_{n=1}^{\infty} \frac1{n^2}\right) \left(\int_Y\;  \sum_{n=1}^{\infty} n^2 \int_X |f_n^{(t)} - f^{(t)}|^2 \, d_mX(x) \, w_m(t)dm_Y(t)\right)^{1/2}\\
    &=  \frac{\pi^2}{6} \left(\sum_{n=1}^{\infty} n^2 \int_Y\;  \int_X   |f_n^{(t)}(x) - f^{(t)}(x)|^2 \, d_mX(x) \, w_m(t)dm_Y(t)\right)^{1/2}\\
    & = \frac{\pi^2}{6} \left(\sum_{n=1}^{\infty} n^2 \left\|f_n-f \right\|_{L^2(X \times_w Y,m_w)}^2 \right)^{1/2} < \infty
\end{align*}

This shows that for $w_mm_Y$-a.e. $t\in Y$ we have $\sum_{n=1}^{\infty} \|f_n^{(t)} - f^{(t)}\|_{L^2(X,m_X)}<\infty$ and so in particular $f_n^{(t)}\rightarrow f^{(t)}$ in $L^2(X, m_X)$. Similarly for $m_X$-a.e. $x\in X$, we have $f_n^{(x)}\rightarrow f^{(x)}$ in $L^2(Y, w_mm_Y)$.

Now observe that for $(x,t)\in X \times_w Y$ we have:
\begin{align}
    \begin{split}
        \lip_{X \times_w Y}(f_n)(x,t) &= \limsup_{(y,s)\rightarrow(x,t)}\frac{|f_n(y,s)-f_n(x,t)|}{d_w((y,s),(x,t))} \\
        &\geq \limsup_{s\rightarrow t}\frac{|f_n(x,s)-f_n(x,t)|}{d_w((x,s),(x,t))} \\
        &= \limsup_{s\rightarrow t}\frac{|f_n^{(x)}(s)-f_n^{(x)}(t)|}{d_Y(s,t)} = \lip_Y(f_n^{(x)})(t).
    \end{split}
\end{align}

Then, by Fatou's lemma:
\begin{align}
    \begin{split}
        &\int_X \liminf_{n\rightarrow \infty} \int_Y \, \lip_Y(f_n^{(x)})^2(t) w_m(t)dm_Y(t) dm_X(x)\\
        &\leq \liminf_{n\rightarrow \infty} \int_{X \times_w Y} \lip_Y(f_n^{(x)})^2(t) dm_w(x,t)\\
        &\leq \liminf_{n\rightarrow \infty} \int_{X \times_w Y} \lip_{X \times_w Y}(f_n)^2(x,t) dm_w(x,t)\\
       & =  \int_{X \times_w Y} |Df|^2_{X \times_w Y} dm_w(x,t) < \infty.
    \end{split}
\end{align}

Since $f_n^{(x)}\rightarrow f^{(x)}$ in $L^2(Y, w_mm_Y)$ for $m_x$-a.e. $x\in X$, the last inequality together with the lower semicontinuity of minimal weak upper gradients gives that $f^{(x)} \in W^{1,2}(Y, w_mm_Y)$ for $m_x$-a.e. $x\in X$ and
\begin{equation}\label{xineq}
     \int_{X \times_w Y} |Df^{(x)}|_Y^2(t) dm_w(x,t)\; \leq \int_{X \times_w Y} |Df|^2_{X \times_w Y} dm_w(x,t).
\end{equation}

With an analogous argument we can get conditions on $f^{(t)}$.
Starting from the bound:
\begin{align}
    \begin{split}
        \lip_{X \times_w Y}(f_n)(x,t) &= \limsup_{(y,s)\rightarrow(x,t)}\frac{|f_n(y,s)-f_n(x,t)|}{d_w((y,s),(x,t))} \\
        &\geq \limsup_{y\rightarrow x}\frac{|f_n(y,t)-f_n(x,t)|}{d_w((y,t),(x,t))} \\
        &= \limsup_{y\rightarrow x}\frac{|f_n^{(t)}(y)-f_n^{(t)}(x)|}{w_d(t)d_X(x,y)} = \frac{1}{w_d(t)} \lip_X(f_n^{(t)})(x).
    \end{split}
\end{align}

This inequality, valid for every $t\in Y$ such that $w_d(t)>0$, grants that $f^{(t)} \in W^{1,2}(X)$ for $w_mm_Y$-a.e. $t\in Y$ (here we are using the assumption that $\{w_d = 0 \} \subset \{w_m = 0 \}$) and that
\begin{equation}\label{tineq}
     \int_{X \times_w Y} \frac{|Df^{(t)}|_X^2(x)}{w_d^2(t)} dm_w(x,t)\; \leq \int_{X \times_w Y} |Df|^2_{X \times_w Y} dm_w(x,t).
\end{equation}
The bounds \eqref{xineq} and \eqref{tineq} ensure that $f \in \BL_w(X,Y)$ so that the desired inclusion is proved.
\end{proof}

\bigskip \begin{lemma}[Lemma 3.11 in \cite{gigli2018sobolev}]\label{comparable}
    Let $X$ be a set, $d_1,\,d_2$ two distances on $X$ and $m_1,\, m_2$ two measures. Assume also that $(X,d_1,m_1)$ and  $(X,d_2,m_2)$ are metric spaces that satisfy Assumption \ref{XAssum} and that for some $C>0$ we have $m_2\leq Cm_1$ and that for some $L>0$ we have $d_1\leq Ld_2$. Then, denoting by $S(X_1)$ and $S(X_2)$ the Sobolev classes relative to $(X,d_1,m_1)$ and $(X,d_2,m_2)$ respectively and by $|Df|_1$ and $|Df|_2$ the associated minimal weak upper gradients, we have $S(X_1)\subset S(X_2)$
    and for every $f\in S(X_1)$ the inequality $|Df|_2\leq L|Df|_1$ holds $m_2$-a.e.
\end{lemma}

\bigskip \begin{proposition}\label{Prop312}
    Let $f \in W^{1,2}(X \times_w Y) \subset \BL_w(X,Y)$. Then $|Df|_{\BL_w} \leq |Df|_{X \times_w Y} \leq C_0 |Df|_{\BL_w}$ $m_w$-a.e., where $C_0$ is as in Theorem \ref{Thm3.7}.
\end{proposition}

\begin{proof}
    Fix $\varepsilon >0$. Let $t_0$ be such that $w_m(t_0) >0$, and hence $w_d(t_0) >0$. By continuity we can find $\delta >0$ such that
    \begin{equation} \label{wdiscont}
        \left|\frac{w_d(t)}{w_d(s)} \right| \leq 1 + \varepsilon \qquad \text{for all } t,s \in \overline{B(t_0,3\delta)}.
    \end{equation}
Let $\chi \colon Y \to [0,1]$ be a Lipschitz function such that $\chi \equiv 1$ on $\overline{B(t_0,\delta)}$, and $\chi \equiv 0$ outside of $\overline{B(t_0,3\delta)}$. Define two continuous functions (in order to have the product being warped only around $t_0$) as follows:
\begin{equation}
\overline{w_d}(t) =
    \begin{cases}
    w_d(t) & \text{if }\dist_Y(t,t_0) \leq 2\delta, \\
    \frac{\dist_Y(t,t_0) - 2\delta}{\delta} w_d(t_0) + \frac{3\delta - \dist_Y(t,t_0)}{\delta}w_d(t) & \text{if } 2\delta < \dist_Y(t,t_0) < 3\delta, \\
    w_d(t_0) & \text{if } \dist_Y(t,t_0) \geq 3\delta.
    \end{cases}
\end{equation}
\begin{equation}
\overline{w_m}(t) =
    \begin{cases}
    w_m(t) & \text{if }\dist_Y(t,t_0) \leq 2\delta, \\
    \frac{\dist_Y(t,t_0) - 2\delta}{\delta} w_m(t_0) + \frac{3\delta - \dist_Y(t,t_0)}{\delta}w_m(t) & \text{if } 2\delta < \dist_Y(t,t_0) < 3\delta, \\
    w_m(t_0) & \text{if } \dist_Y(t,t_0) \geq 3\delta.
    \end{cases}
\end{equation}
and let $(X \times_{\overline{w}} Y, \dist_{\overline{w}},m_{\overline{w}})$ be the corresponding product space. Consider the function $\overline{f} \colon X \times_w Y \to \R$ defined by $\overline{f}(x,t) = \chi(t)f(t,x)$. Clearly $\overline{f} \in W^{1,2}(X \times_w Y)$, and hence to $\BL_w(X,Y)$. Since minimal upper gradients are local we know that
\[
|Df|_{X \times_w Y}=|D\overline{f}|_{X \times_w Y} \quad \text{and} \quad |Df|_{\BL_w}=|D\overline{f}|_{\BL_w} \quad \text{ $m_w$-a.e. on $X \times \overline{B(t_0,\delta)}$}.
\]
Because $\overline{f}$ is supported on $B(t_0,3\delta) \times X$, where $w_d$ is positive we can think of $\overline{f}$ as a function on $X \times_{\overline{w}} Y$. With this identification we have
\[
|D\overline{f}|_{X \times_w Y}=|D\overline{f}|_{X \times_{\overline{w}} Y} \quad \text{and} \quad |D\overline{f}|_{\BL_w}=|D\overline{f}|_{\BL_{\overline{w}}} \quad \text{ $m_w$-a.e. on $X \times \overline{B(t_0,2\delta)}$}.
\]

We now want to use this ``localized warped space'' to utilize the results for the Cartesian product we have obtained in the previous section. 

Consider first the space $\overline X := (X, w_d(t_0) \dist_X, w_m(t_0) m_X)$. This is simply our original metric measure space $(X, \dist_X, m_X)$ which has been re-scaled. Now consider $(\overline X \times Y, \dist, m)$ where with a slight abuse of notation we denote by $\dist$ and $m$ are the appropriate Cartesian metric and measure, respectively. Set $c= \min_{\overline{B(t_0,3\delta)}} w_m$ and $C= \max_{\overline{B(t_0,3\delta)}} w_m$. From the definition of $m_w$ we immediately have 
\[
c m \leq m_{\overline{w}} \leq C m,
\]
and, recalling (\ref{wdiscont}), we have 
\begin{equation}
    (1 + \varepsilon)^{-1} \dist \leq \dist_{\overline{w}} \leq (1 + \varepsilon)\dist.
\end{equation}
Let us denote by $\overline \BL := \BL(\overline X,Y)$. Using Lemma \ref{comparable}, recalling that $X \times_{\overline{w}} Y$ and $\overline X \times Y$ coincide as sets, we obtain
\[
 (1 + \varepsilon)^{-1} |D\overline{f}|_{\overline X \times Y} \leq |D\overline{f}|_{X \times_{\overline{w}}Y} \leq (1 + \varepsilon)|D\overline{f}|_{\overline X \times Y},
 \]
and  
\[(1 + \varepsilon)^{-1} |D\overline{f}|_{\overline \BL} \leq |D\overline{f}|_{\BL_{\overline{w}}} \leq (1 + \varepsilon)|D\overline{f}|_{\overline \BL}.
\]

We now exploit the Cartesian results applied to the pair $(\overline X,Y)$: by Theorem \ref{Thm3.7}, we know that 
\[
|D\overline{f}|_{\overline \BL}\leq|D\overline{f}|_{\overline X \times Y} \leq C_0 |D\overline{f}|_{\overline \BL} \quad m\text{-a.e.}.
\]
 Putting everything together we obtain
\[
(1+\varepsilon)^{-2} |Df|_{\BL_{\overline{w}}} \leq |Df|_{X \times_{\overline{w}}Y} \leq (1+ \varepsilon)^2 C_0 |Df|_{\BL_{\overline{w}}}
\]
$m_w$-a.e. on $X \times B(t_0,\delta)$. Because $t_0$ was arbitrary, and because every cover of $\{w_m >0\} \subset Y$ has a countable subcover, the conclusion holds by letting $\varepsilon \to 0$.
\end{proof}

\bigskip \begin{proposition} \label{prop8}
We have $\BL_{0,w}(X,Y) \subset W^{1,2}(X \times_w Y)$.
\end{proposition}

\begin{proof}
    Thanks to Proposition \ref{Prop312}, it is sufficient to prove that $\mathcal{V}\subset W^{1,2}(X \times_x Y)$, where $\mathcal{V}$ is defined as in Definition \ref{Vset}. Fix $k \in \N$, let $t_0 \in Y$, and let 
    \begin{equation} \label{cutoff}
    \psi_{n,k}(t) = \sum_{\substack{j \\ \supp \chi_{j,k} \cap B(t_0,n) \neq \varnothing}} \chi_{j,k}(t),
    \end{equation}
    where the functions $\chi_{i,k}$ are the partition of unity subordinate to the covering by ``cubes'' from Lemma \ref{cubes}. Observe that $\psi_{k,n} \equiv 1$ inside $B(t_0,n-1)$. With a slight abuse of notation we will denote $\psi_{n,k}$ simply as $\psi_n$, as $k$ will remained fixed throughout the proof. For $f\in \BL_w(X,Y)$ we define $f_n(x,t):= \psi_n(t)f(x,t)$, and note that by definition $f_n\in \BL_w(X,Y)$, while by the dominated convergence theorem and inequality (\ref{ProdRule}) we obtain $f_n\rightarrow f$ in $\BL_w(X,Y)$.
    
    The idea of the proof is as follows: we show that any $f\in \mathcal{V}$ with support contained in $Y\cap B(t_1,R)$ for $R>0$ and $t_1 \in Y$ belongs to $W^{1,2}(X \times_w Y)$. This together with Proposition \ref{Prop312}, which ensures $\BL$-convergence implies $W^{1,2}$-convergence, will complete the proof.

    Accordingly, fix such $f\in \mathcal{V}$ and for $r\in(0,1)$ let  $\Omega_r\subset Y$ be the r-neighborhood of $\{w_m=0\}$. Now, find $r\in(0,1)$ such that $f$ is $m_w$-a.e. zero on $X\times \Omega_{2r}$. Then, recalling that $\{w_d=0\}\subset\{w_m=0\}$ and by continuity and compactness we have that there exist constants $0<c<C<\infty$ such that
    \[
    c<w_d(t),\; w_m(t)<C, \quad \forall t\in  Y\cap B(t_1,R) \setminus \Omega_{\frac{r}{2}}.
    \]

    We now use a comparison argument similar to the one used in Proposition \ref{Prop312}. Let $w'_d$ and $w'_m$ two continuous functions which agree with $w_d$ and $w_m$ on $B(t_1,R) \setminus \Omega_{\frac{r}{2}}$ and such that $c<w'_d(t),\; w'_m(t)<C$ on the whole $Y$. Consider now the warped product $(X \times_{w'} Y, d_{w'}, m_{w'})$ and Cartesian product of $(X \times Y, d, m)$ of $X$ and $Y$. Then by Lemma \ref{comparable} and the properties of $w'_d,\; w'_m$ we have the following equalities of sets:
    \begin{equation}\label{set1}
        \BL_{w'}(X, Y) = \BL(X,Y) \quad \text{and} \quad W^{1,2}(X \times_{w'} Y) = W^{1,2}(X \times Y).
    \end{equation}
    Moreover by Theorem \ref{Thm3.7} the following equality of sets also holds:
     \begin{equation}\label{set2}
        \BL(X,Y) = W^{1,2}(X \times Y).
    \end{equation}
    Finally, putting together \eqref{set1} and \eqref{set2} we obtain:
    \begin{equation}\label{set3}
        \BL_{w'}(X,Y) =\BL(X,Y) = W^{1,2}(X \times Y)=W^{1,2}(X \times_{w'} Y).
    \end{equation}
    By construction of $w'_d,\; w'_m$ we have that $f\in \BL_{w'}(X,Y)$ so that, by equation \eqref{set3}, $f\in W^{1,2}(X \times_{w'} Y)$. By density in energy of Lipschitz functions (Lemma \ref{densitylip}) there exists a sequence of $d_{w'}$-Lipschitz functions $f_n$ that converges to $f$ in $L^2(X \times_{w'} Y)$ and we also have
    \[
    \sup_{n\in \mathbb{N}} \int \lip_{X \times_{w'} Y}(f_n)^2\, dm_{w'} < \infty
    \]
    uniformly bounded in $n$. From now on we assume $f_n$ is bounded for every $n\in \mathbb{N}$. This is possible up to replacing the original $f_n$ with $\min(\max(f_n,-C_n), C_n)$ for sufficiently large $C_n$.

    Now, similarly to the beginning of the proof, we find a Lipschitz function $\psi \colon Y\rightarrow [0,1]$ which is identically 0 on $\Omega_r \cup (Y\setminus B(t_1, R+1))$ and identically 1 on $Y \cap B(t_1, R) \setminus \Omega_{2r}$ and set $\Tilde{f}_n(t,x) := \psi(t) f_n(t,x)$.
    By construction the functions $\Tilde{f}_n$ are still $d_{w'}$-Lipschitz, they converge to $f\in L^2(X \times_{w'} Y)$ and satisfy
    \begin{equation}\label{sup'}
        \sup_{n\in \mathbb{N}} \int \lip_{X \times_{w'} Y}(\Tilde{f}_n)^2\, dm_{w'} < \infty.
    \end{equation}

    To conclude the proof we need to show that the functions $\Tilde{f}_n$ are also $d_{w}$-Lipschitz, converge to $f\in L^2(X \times_{w} Y)$ and satisfy
    \begin{equation}\label{sup}
     \sup_{n\in \mathbb{N}} \int \lip_{X \times_{w} Y}(\Tilde{f}_n)^2\, dm_{w} < \infty.
     \end{equation}

    First, $\Tilde{f}_n$ converges to $f\in L^2(X \times_{w} Y)$ since the functions $\Tilde{f}_n$ and $f$ are concentrated on $X \times (B(t_1,R) \setminus \Omega_r)$ and on this set the measures $m_w$ and $m_{w'}$ agree by definition. Moreover, since $w_d$ and $w_{d'}$ agree on $X \times (B(t_1,R) \setminus \Omega_r)$, by definition we also have 
    $$
    \lim_{(y,s)\rightarrow(x,t)} \frac{d_w((y,s), (x,t))}{d_{w'}((y,s), (x,t))} = 1 \quad \forall \; (x,t) \in X \times (B(t_1,R) \setminus \Omega_r),
    $$
    so that, in particular, $\lip_{X \times_{w} Y}(\Tilde{f}_n) = \lip_{X \times_{w'} Y}(\Tilde{f}_n)$. This, together with \eqref{sup'}, proves \eqref{sup}.

    It remains to prove that $\Tilde{f}_n$ are $d_{w}$-Lipschitz, that is, we need to prove that $\Lip(\Tilde{f}_n)<\infty$. Recall that, since we are on a length space, the Lipschitz constant of a function is equal to the supremum of the local Lispchitz constants. Then, by denoting $\Lip'(\Tilde{f}_n)$ the $d_{w'}$-Lipschitz constant and recalling that by construction of $\Tilde{f}_n$, $\Lip'(\Tilde{f}_n) < \infty$, we have:
    
    \[
    \Lip(\Tilde{f}_n) = \sup_{X \times_{w} Y} \lip_{X \times_{w} Y}(\Tilde{f}_n) = \sup_{X \times_{w'} Y} \lip_{X \times_{w'} Y}(\Tilde{f}_n) = \Lip'(\Tilde{f}_n) < \infty.
    \]
    Now, the conclusion follows from the lower semicontinuity of weak upper gradients \eqref{lsc}, and the bound $|D\Tilde{f}_n|_{X \times_{w} Y} \leq \lip_{X \times_{w} Y}(\Tilde{f}_n)$ that is valid $m_w$-a.e.. 
\end{proof}

By adding a few extra assumptions on the function $w_m$ we can improve the previous result and show that, with these additional assumptions, the inclusions above are all equalities.

\bigskip \begin{proposition}
    Assume that $w_m$ satisfies
\begin{gather}
\{w_m=0\} \subset Y \text{ is discrete;} \\
\text{$w_m$ decays at least linearly near its zeros, i.e. } \\
w_m(t) \leq C \inf_{ s \mid w_m(s)=0} \dist_Y(t,s), \quad \forall t \in Y. \nonumber
\end{gather}       
Also assume that $m_Y$ is an upper regular measure, that is, there exists a constant $c$ such that $m_Y(B(y,r))\leq cr$ for $y \in Y$, and $0 < r < 1$. Then $\BL_{0,w}(X,Y) =W^{1,2}(X \times_w Y)=\BL_{w}(X,Y)$.
\end{proposition}

\begin{proof}
    It is easy to see that $\BL_w(X,Y) \cap L^{\infty}(X \times_w Y)$ is dense in $\BL_w(X,Y)$, using a standard truncation argument, so that it is enough to show that for any $f \in \BL_w(X,Y) \cap L^{\infty}(X \times_w Y)$ there is a sequence in $\mathcal{V}$ that converges to it in $\BL_w(X,Y)$. 

    Pick $f \in \BL_w(X,Y) \cap L^{\infty}(X \times_w Y)$ and define $D(t)= \min_{s \mid w_m(s)=0} \dist_Y(t,s)$. For $m,n \in \N$, with $n>1$, fix $x_0 \in X$ and $t_0 \in Y$, and let $\psi_{n,k}(t)$ be as in (\ref{cutoff}). Moreover define
    \begin{align*}
        \sigma_m(x) & = \max \{0,\min\{m - \dist_X(x,x_0),1\}\},\\
        \eta_l (t) & = \max \{0,\min\{1 + \frac{\log(D(t))}{\log l},1\}\}.
    \end{align*}
    Define $f_{n,m,k,l}(x,t)= \psi_{n,k}(t)\eta_l (t)\sigma_m(x)f(x,t)$. Because the product of the three auxiliary functions is Lipschitz and bounded for all $n,m,k,l$, $f_{n,m,k,l} \in \BL_w(X,Y)$, and because $\eta_l$ is $0$ in a neighborhood of $\{w_m=0\}$, we also have $f_{n,m,k,l} \in \mathcal{V}$, for all $n,m,k,l$.

By construction, the function $(x,t) \mapsto \psi_{n,k}(t)\eta_l (t)\sigma_m(x)$ is uniformly bounded by $1$, and hence by the dominated convergence theorem we have that $f_{n,m,k,l} \to f$ in $L^2(X \times_w Y)$, as $n,m,l \to \infty$, for every $k$. 

Now, recalling (\ref{ProdRule}), and because $\sigma_m$ is $1$-Lipschitz, for $m_w$-a.e. $(x,t)$
\[
|\partial/\partial x (f-f_{n,m,k,l})(x,t)| \leq |\psi_{n,k}(t)\eta_l (t)\sigma_m(x)-1| |\partial f/\partial x|(x,t) + |f(x,t)| 1_{\{\dist_X(\cdot,x_0) \geq m-1\}}(x).
\]
Applying the dominated convergence theorem again we get that, as $n,m,l \to \infty$, 
\[
\int_{X \times Y} |\partial/\partial x (f-f_{n,m,k,l})(x,t)|^2 dm_w \to 0.
\]
Using again Lipschitzness and the product rule we have
\begin{align*} 
|\partial/\partial t (f-f_{n,m,k,l})(x,t)| & \leq |\psi_{n,k}(t)\eta_l (t)\sigma_m(x)-1| |\partial f/\partial t|(x,t) \\
& \qquad+ c_1C_Y^3k|f(x,t)|1_{\{\dist_Y(\cdot,t_0) \geq n-1\}}(t) \\
& \qquad+ |f(x,t)| 1_{\{\dist_X(\cdot,x_0) \leq m\}}(x)1_{\{\dist_Y(\cdot,t_0) \leq n\}}(t) |\partial \eta_l /\partial t |(t),
\end{align*}

for $m_w$-a.e. $(x,t)$. Once again by the dominated convergence theorem, the first two terms go to $0$ in $L^2$ as $n,m,l \to \infty$. For the last term, observe first that $ |\partial \eta_l /\partial t |(t) \leq \frac{1_{D^{-1}([l^{-1},1])}(t)}{D(t)\log l }$. Now we can use both the additional assumptions on $w_m$: if $y_1, \dots, y_N$ are the finite number of zeroes of $w_m$ inside $\overline{B(t_0,n-1)}$, and recalling that $f$ is bounded, we have
\begin{align*} 
\int_{X \times Y} |f(x,t)|^2 & 1_{\{\dist_X(\cdot,x_0) \leq m\}}(x)1_{\{\dist_Y(\cdot,t_0) \leq n\}} |\partial \eta_l /\partial t |^2(t) \,dm_w(x,t) \\
& \leq \frac{\|f\|^2_{L^{\infty}} m_X (B(x_0,m))}{(\log l)^2} \int_{B(t_0,n) \cap D^{-1}([l^{-1},1])} \frac{1}{D(t)^2} w_m(t)\,dm_Y(t) \\ 
& \leq C \frac{\|f\|^2_{L^{\infty}} m_X (B(x_0,m))}{(\log l)^2} \int_{B(t_0,n) \cap D^{-1}([n^{-1},1])} \frac{1}{D(t)} \,dm_Y(t) \\ 
& \leq C \frac{\|f\|^2_{L^{\infty}} m_X (B(x_0,m))}{(\log l)^2} \sum_{i=1}^N \int_{\{t \mid \dist_Y(y_i,t) \in [l^{-1},1]\}} \frac{1}{\dist_Y(t,y_i)}\,dm_Y(t) \\
& \leq C \frac{\|f\|^2_{L^{\infty}} m_X (B(x_0,m))}{(\log l)^2} \sum_{i=1}^N \sum_{j=1}^{J} \int_{\{t \mid \dist_Y(y_i,t) \in [2^{-j},2^{-j+1}]\}} \frac{1}{\dist_Y(t,y_i)}\,dm_Y(t) \\
& \leq C \frac{\|f\|^2_{L^{\infty}} m_X (B(x_0,m))}{(\log l)^2} \sum_{i=1}^N \sum_{j=1}^{J} \int_{\{t \mid \dist_Y(y_i,t) \in [2^{-j},2^{-j+1}]\}} 2^j \,dm_Y(t) \\
& \leq C \frac{\|f\|^2_{L^{\infty}} m_X (B(x_0,m))}{(\log l)^2} \sum_{i=1}^N \sum_{j=1}^{J} 2^j m_Y(B(y_i,2^{-j+1})) \\
& \leq Cc \frac{\|f\|^2_{L^{\infty}} m_X (B(x_0,m))}{(\log l)^2} \sum_{i=1}^N \sum_{j=1}^{J} 2^j 2^{-j+1} \\
& \leq 2NCc \frac{\|f\|^2_{L^{\infty}} m_X (B(x_0,m))}{\log l}.
\end{align*}
where we defined $J$ as the smallest integer such that $2^J \geq l$. The last term goes to $0$ as $l\to \infty$ for every $m,m,k \in \N$, and so we have proved the desired result.
\end{proof}

\textbf{Acknowledgments}\\
The project started while the authors were participating at the AMS MRC 2020 on Analysis in Metric Spaces.  We are grateful to Nicola Gigli, Nageswari Shanmungalingam, and Luca Capogna for their helpful guidance. We are also grateful to Angela Wu for useful insights in the early stages of the project.

\textbf{Funding}\\
 S.G. was partially supported by NSF DMS-FRG-1853993. V.G. has been supported by NSF DMS-FRG-1854344. E.N. is supported by the Simons Postdoctoral program at IPAM, NSF DMS 1925919 and NSF-2331033.

\bibliographystyle{alpha} 
\bibliography{sn-bibliography}

\end{document}